\documentclass[a4paper,10pt, reqno]{amsart}
\usepackage{amssymb,latexsym,amsmath,epsfig,amsthm} 

\theoremstyle{plain}
\newtheorem{theorem}{Theorem}
\newtheorem*{theorem*}{Theorem}
\newtheorem{corollary}{Corollary}
\newtheorem*{corollary*}{Corollary}
\newtheorem{lemma}{Lemma}
\newtheorem*{lemma*}{Lemma}

\newtheorem*{proposition*}{Proposition}

\newtheorem*{conjecture*}{Conjecture}
\theoremstyle{definition}
\newtheorem{definition}{Definition}
\newtheorem*{definition*}{Definition}
\theoremstyle{remark}
\newtheorem{remark}{Remark}
\newtheorem*{remark*}{Remark}

\begin{document}
\title[Fractal sets]{On one class of fractal sets}
\author{Symon Serbenyuk}
\address{Institute of Mathematics \\
 National Academy of Sciences of Ukraine \\
  3~Tereschenkivska St. \\
  Kyiv \\
  01004 \\
  Ukraine}
\email{simon6@ukr.net}



\begin{abstract}
 In the present article a new  class  $\Upsilon$ of all sets  represented in the following form is introduced:
$$
\mathbb S_{(s,u)}\equiv\left\{x: x= \Delta^{s}_{{\underbrace{u...u}_{\alpha_1-1}} \alpha_1{\underbrace{u...u}_{\alpha_2 -1}}\alpha_2 ...{\underbrace{u...u}_{ \alpha_n -1}}\alpha_n...},  \alpha_n \in A_0, \alpha_n \ne u, \alpha_n \ne 0 \right\}, 
$$
where $2<s\in \mathbb N$ and $u\in A$  are fixed parameters, $A=\{0,1,\dots,s-1\}$, and $A_0=A \setminus \{0\}$. Topological, metric, and fractal properties of these sets are studied. 
The problem of belonging to these sets of normal numbers is investigated. 
The theorem on the calculation of the Hausdorff-Besicovitch dimension of an arbitrary set whose elements have restrictions on using of  digits or combinations of digits in own s-adic representations is formulated and proved.

\end{abstract}
\maketitle



\section{Introduction}

A fractal in the wide sense is a set whose the topological dimension and the Hausdorff-Besicovitch dimension are different, and in the narrow sense is a set that has the fractional Hausdorff-Besicovitch dimension. In 1977, the notion  ``fractal" was introduced by B.~Mandelbrot in  \cite{Mandelbrot1977} but such sets one can to see in earlier mathematics researches. For example, the Cantor set 
$$
C[3, \{0,2\}]=\{x: x=\Delta^3 _{\alpha_1\alpha_2\ldots \alpha_n\ldots}, \alpha_n\in\{0,2\}\}
$$
 is a fractal. Fractals are applicated in modeling of  biological and physical processes,  and the notion of ``fractals" unites various mathematical objects such as
continuous nowhere monotonic functions, nowhere differentiable functions, singular distributions, etc. \cite{{Falconer1997}, {Falconer2004}}.

 In the present article a new class of certain fractal sets whose elements have a functional  restriction on using of symbols in own 
s-adic representations is introduced, and topological, metric, and fractal properties of such sets are studied. In 2012, the results of this article were presented by the author in the  abstracts
\cite{{S. Serbenyuk abstract 2}, {S. Serbenyuk abstract 3}, {S. Serbenyuk abstract 5}}. 

An expansion of a real number $x_0 \in [0;1]$ in the form
\begin{equation}
\label{eq: s-adic expansion}
x_0=\frac{\alpha_1}{s}+\frac{\alpha_2}{s^2}+\dots+\frac{\alpha_n}{s^n}+\dots ,
\end{equation}
is called \emph{the s-adic expansion of $x_0$}. Here $1<s$ is a fixed positive integer, $\alpha_n \in A=\{0, 1,\dots , s-1\}$.
By $x_0=\Delta^s _{\alpha_1\alpha_2...\alpha_n...}$ denote the s-adic expansion of $x_0$. The  notation $\Delta^s _{\alpha_1\alpha_2...\alpha_n...}$ is called \emph{the s-adic representation of $x_0$}.

A number whose the s-adic representation is periodic with the period $(0)$ is called \emph{an s-adic-rational number}. Any s-adic-rational number has two different s-adic representations, i.e., 
$$
\Delta^s _{\alpha_1\alpha_2...\alpha_{n-1}\alpha_n(0)}=\Delta^s _{\alpha_1\alpha_2...\alpha_{n-1}[\alpha_n-1](s-1)}.
$$
The other numbers in $[0;1]$ are called \emph{s-adic-irrational numbers.} These numbers have the unique s-adic representation.

\section{The object of research}

Let  $2<s$  be a fixed positive integer, $A=\{0,1,...,s-1\}$ be an alphabet of the s-adic number system, 
$A_0=A \setminus \{0\}=\{1,2,\dots , s -1\}$,  and
$$
 L \equiv  (A_0)^{\infty}= (A_0) \times  (A_0) \times  (A_0)\times\dots  
$$
be the space of one-sided sequences of  elements of $ A_0$.

Consider a new class  $\Upsilon_s$ of sets  $\mathbb  S_{(s,u)}$ represented  in the form 
\begin{equation}
\label{S(s,u)1}
\mathbb S_{(s,u)}\equiv\left\{x: x= \Delta^{s}_{{\underbrace{u...u}_{\alpha_1-1}} \alpha_1{\underbrace{u...u}_{\alpha_2 -1}}\alpha_2 ...{\underbrace{u...u}_{ \alpha_n -1}}\alpha_n...},  (\alpha_n) \in L, \alpha_n \ne u, \alpha_n \ne 0 \right\}, 
\end{equation}
where $u=\overline{0,s-1}$, $u$ and $s$ are fixed for the set $\mathbb  S_{(s,u)}$. That is the class $\Upsilon_s$ contains the sets  $\mathbb  S_{(s,0)}, \mathbb  S_{(s,1)},\dots,\mathbb  S_{(s,s-1)}$. We say that   $\Upsilon$ is a class of sets such that contains the classes   $\Upsilon_3, \Upsilon_4,\dots ,\Upsilon_n,\dots$.

It is easy to see that the set  $\mathbb  S_{(s,u)}$ can be defined by the s-adic expansion in the following form
\begin{equation}
\label{S(s,u)2}
\mathbb S_{(s,u)}\equiv \left\{x: x=\frac{u}{s-1} +\sum^{\infty} _{n=1} {\frac{\alpha_n - u}{s^{\alpha_1+\dots+\alpha_n}}}, (\alpha_n) \in L, \alpha_n \ne u, \alpha_n \ne 0  \right\}.
\end{equation}

\section{Topological and metric properties of  $\mathbb  S_{(s,0)}$}

By $x_0=\Delta^{'} _{c_1c_2...c_n...}$ denote the following  equality
$$
\mathbb  S_{(s,0)}\ni x_0 =\frac{c_1}{s^{c_1}}+\frac{c_2}{s^{c_1+ c_2}}+\frac{c_3}{s^{c_1+ c_2+c_3}}+\dots+\frac{c_n}{s^{c_1+ c_2+\dots+c_n}}+\dots .
$$

 It follows from the definition of the set  $\mathbb  S_{(s,0)}$ that s-adic-rational numbers do not belong to $\mathbb  S_{(s,0)}$. Hence  each element of $\mathbb  S_{(s,0)}$ has the unique s-adic representation.
\begin{lemma}
The set $\mathbb  S_{(s,0)}$  is a uncountable set.
\end{lemma}
\begin{proof} Let  the mapping $f: \mathbb  S_{(s,0)} \to C[s,A_0]$, where $C[s,A_0]=\{x: x=\Delta^s _{\alpha_1\alpha_2\ldots\alpha_n\ldots}, \alpha_n\in~A_0\}$,  be given by
\begin{equation}
\label{form1}
\forall (\alpha_n)\in L:x=\sum^{\infty} _{n=1} {\frac{\alpha_n}{s^{\alpha_1+ \alpha_2+...+\alpha_n}}}  \to \sum^{\infty} _{n=1} {\frac{\alpha_n}{s^n}}=y=f(x),
\end{equation}
i.e.,
\begin{equation}
\label{form2}
x= \Delta^ {s}_{\underbrace{0...0}_{c_1 - 1}c_1 \underbrace{0...0}_{c_2 - 1}c_2...\underbrace{0...0}_{c_n - 1}c_n...} \to \Delta^{s} _{c_1c_2...c_n...}=y=f(x).
\end{equation}
From (\ref{form2}) it follows that (\ref{form1}) is a bijection. Since $ C[s,A_0] $ is a  uncountable set, we see that  $\mathbb  S_{(s,0)}$ is a  uncountable set.
\end{proof}

To investigate topological and metric properties of  $\mathbb  S_{(s,0)}$, we shall study properties of cylinders. 

\begin{definition}

  \emph{A cylinder $ \Delta' _{c_1c_2...c_n} $ of rank   $ n $ with base $ c_1 c_2\ldots c_n $  } is a set formed by all numbers of $\mathbb  S_{(s,0)}$ with s-adic representations in which the first $ n $  non-zero digits are fixed and coincide with
$ c_1,c_2,\dots ,c_n, $ respectively.
\end{definition}
\begin{lemma}
\label{lemma2}
Cylinders have the following properties:
\begin{enumerate}
\item \label{1}
$$
\inf \Delta^{'} _{c_1c_2...c_n}=\Delta^{'} _{c_1c_2...c_n(s-1)}=g_n+\frac{s-1}{s^{c_1+c_2+\dots+c_n}(s^{s-1}-1)},
$$
$$
\sup  \Delta' _{c_1c_2...c_n}=\Delta^{'} _{c_1c_2...c_n(1)}=g_n+\frac{1}{s^{c_1+c_2+\dots+c_n}(s-1)},
$$
$$
\mbox{where} ~g_n=\sum _{k=1} ^n \frac{c_k}{s^{c_1+c_2+\dots+c_k}}.
$$
\item \label{2} 
$$
 \Delta' _{c_1c_2...c_n} \subset I_{c_1c_2...c_n},
$$
 where  endpoints of the segment $ I_{c_1c_2\ldots c_n}$ coincide with  endpoints of  $\Delta^{'} _{c_1c_2\ldots c_n}$, i.e., 
\begin{equation}
    \label{eq:  000}
I_{c_1c_2...c_n}=\left[g_n+\frac{s-1}{(s^{s-1}-1)s^{c_1+c_2+\dots+c_n}};g_n+\frac{1}{(s-1)s^{c_1+c_2+\dots+c_n}}\right].
\end{equation}
\item \label{3}
$$
\Delta' _{c_1c_2...c_n}\subset \Delta^s _{\underbrace{0...0}_{c_1 - 1}c_1 \underbrace{0...0}_{c_2 - 1}c_2...\underbrace{0...0}_{c_n - 1}c_n},~ c_i \in A_0.
$$
\item \label{4}
$$
  \Delta^{'} _{c_1c_2...c_n} =\bigcup^{s-1} _{i=1} { \Delta^{'} _{c_1c_2...c_ni}}~~~\forall c_n \in A_0,~~~n \in \mathbb N.
$$
\item \label{5}
Suppose that $d(\cdot)$ is the diameter of a set. Then 
$$
d(\Delta' _{c_1c_2...c_n})=\frac{s^{s-1}-1-(s-1)^2}{(s-1)(s^{s-1}-1)s^{c_1+c_2+\dots+c_n}}.
$$
\item \label{6}
$$
\frac{d(\Delta' _{c_1c_2...c_{n+1}})}{d(\Delta' _{c_1c_2...c_n})}=\frac{1}{s^{c_{n+1}}}.
$$
\item \label{7} The following inequality 
$$
\inf \Delta' _{c_1c_2...c_{n-1}p} > \sup  \Delta' _{c_1c_2...c_{n-1} {[p+1]}}
$$
holds for all $p \in N^1 _{s-2}=\{1,2,...,s-2\}$.
\item \label{8}
An arbitrary interval of the form
$$
T_{c_1...c_{n-1}p}=(\sup \Delta' _{c_1...c_{n-1}[p+1]};\inf \Delta'_{c_1...c_{n-1}p})=(b_{p+1};a_p),
$$
where 
$$
b_{p+1}=\frac{p+1}{s^{c_1+\dots+c_{n-1}+p+1}}+\frac{1}{(s-1)s^{c_1+\dots+c_{n-1}+p+1}}+\sum _{k=1} ^{n-1} \frac{c_k}{s^{c_1+c_2+\dots+c_k}}
$$
and
$$
a_p=\frac{p}{s^{c_1+\dots+c_{n-1}+p}}+\frac{s-1}{(s^{s-1}-1)s^{c_1+dots+c_{n-1}+p}}+ \sum _{k=1} ^{n-1} \frac{c_k}{s^{c_1+c_2+dots+c_k}},
$$
satisfies the condition $T_{c_1...c_{n-1}p}\cap \mathbb  S_{(s,0)}=\varnothing$.
\item \label{9}
$$
\Delta^{'} _{c_1c_2...c_{n-1}p} \cap \Delta^{'} _{c_1c_2...c_{n-1}[p+1]}= \varnothing  ~\forall p\in N^1 _{s-2}.
$$
\item \label{10}
If $ x_0 \in \mathbb  S_{(s,0)}$, then
$$
x_0=\bigcap_ {n=1}^{\infty} {\Delta' _{c_1c_2...c_n}}.
$$
\end{enumerate}
\end{lemma}
\begin{proof}
Prove that \emph{the first property} is true. Since  $\frac{a}{s^a}>\frac{b}{s^b}$ holds for  $a<b$, we have
$$
\inf \Delta^{'} _{c_1c_2...c_n}=\frac{1}{s^{c_1+c_2+\dots+c_n}}\left(\frac{s - 1}{s^{s - 1}}+\frac{s-1}{s^{2(s-1)}}+\dots\right)+\sum _{k=1} ^n \frac{c_k}{s^{c_1+c_2+\dots+c_k}}=
$$
$$
=\left(\sum _{k=1} ^n \frac{c_k}{s^{c_1+c_2+...+c_k}}\right)+\frac{s-1}{s^{c_1+c_2+\dots+c_n}(s^{s-1}-1)}=\Delta^{'} _{c_1c_2...c_n(s-1)},
$$
$$
\sup  \Delta' _{c_1c_2...c_n}
=\frac{1}{s^{c_1+c_2+\dots+c_n}}\left( \frac{1}{s}+\frac{1}{s^2}+\dots\right) +\sum _{k=1} ^n \frac{c_k}{s^{c_1+c_2+\dots+c_k}}=
$$
$$
=\left(\sum _{k=1} ^n \frac{c_k}{s^{c_1+c_2+\dots+c_k}}\right)+\frac{1}{s^{c_1+c_2+\dots+c_n}(s-1)}=\Delta^{'} _{c_1c_2...c_n(1)}.
$$
\emph{The second property} follows from Property 1.  \emph{Property 3} follows from the first and the second properties, and  
$$
\Delta^{s} _{\underbrace{0...0}_{c_1 - 1}c_1 \underbrace{0...0}_{c_2 - 1}c_2...\underbrace{0...0}_{c_n - 1}c_n}=\left[\sum _{k=1} ^n \frac{c_k}{s^{c_1+c_2+\dots+c_k}};\frac{1}{s^{c_1+c_2+\dots+c_n}}+\sum _{k=1} ^n \frac{c_k}{s^{c_1+c_2+\dots+c_k}}\right].
$$

\emph{Properties 4-6} follow from the last properties.

Let us prove that \emph{Property 7} is true. Consider the difference 
$$
\inf \Delta' _{c_1c_2...c_{n-1}p} - \sup  \Delta' _{c_1c_2...c_{n-1} {[p+1]}}=\left(\sum _{k=1} ^{n-1} \frac{c_k}{s^{c_1+c_2+\dots+c_k}}\right)+\frac{p}{s^{c_1+c_2+\dots+c_{n-1}+p}}+
$$
$$
+\frac{1}{s^{c_1+c_2+...+c_{n-1}+p}}\cdot \frac{s-1}{s^{s-1}-1}-\left(\sum _{k=1} ^{n-1} \frac{c_k}{s^{c_1+c_2+\dots +c_k}}\right)-\frac{p+1}{s^{c_1+c_2+\dots+c_{n-1}+p+1}}-
$$
$$
-\frac{1}{s^{c_1+c_2+\dots+c_{n-1}+p+1}}\cdot \frac{1}{s-1}
=\frac{(s^{s-1}-1)(p(s-1)^2-s)+s(s-1)^2}{(s-1)(s^{s-1}-1)s^{c_1+c_2+\dots+c_{n-1}+p+1}}>0.
$$

For proving \emph{Property 8}, let us show that the following inequalities hold. 
$$
1) \inf \Delta' _{c_1...c_{n-1}p}- \inf \Delta' _{c_1c_2...c_{n-1}pc_{n+1}}\le 0, ~ p \in N_{s-2} ^1.
$$
$$
2)  \sup \Delta' _{c_1...c_{n-1} {[p+1]}}-\sup \Delta' _{c_1c_2...c_{n-1}[p+1]c_{n+1}} \ge 0, ~ p \in N_{s-2} ^1.
$$
$$
1) \left( \sum _{k=1} ^{n-1} \frac{c_k}{s^{c_1+c_2+\dots+c_k}}\right)+\frac{p}{s^{c_1+\dots+c_{n-1}+p}}+\frac{s-1}{(s^{s-1}-1)s^{c_1+\dots+c_n+c_{n-1}+p}}-
$$
$$
-\left(\sum _{k=1} ^{n-1} \frac{c_k}{s^{c_1+c_2+...+c_k}}\right)-\frac{p}{s^{c_1+\dots+c_{n-1}+p}}-\frac{c_{n+1}}{s^{c_1+\dots+c_{n-1}+p+c_{n+1}}}-
$$
$$
-\frac{1}{s^{c_1+\dots+p+c_{n+1}}}\cdot \frac{s-1}{s^{s-1}-1}=\frac{(s-1)(s^{c_{n+1}}-1)-c_{n+1}(s^{s-1}-1)}{(s^{s-1}-1)s^{c_1+\dots+c_{n-1}+p+c_{n+1}}} \le 0.
$$
This inequality is an equality whenever $c_{n+1}=s-1$.

 Similarly,
$$
2) \left(\sum _{k=1} ^{n-1} \frac{c_k}{s^{c_1+c_2+\dots+c_k}}\right)+\frac{p+1}{s^{c_1+\dots+c_{n-1}+p+1}}+\frac{1}{(s-1)s^{c_1+\dots+c_n+c_{n-1}+p+1}}-
$$
$$
-\left(\sum _{k=1} ^{n-1} \frac{c_k}{s^{c_1+c_2+\dots+c_k}}\right)-\frac{p+1}{s^{c_1+\dots+c_{n-1}+p+1}}-\frac{c_{n+1}}{s^{c_1+\dots+c_{n-1}+p+1+c_{n+1}}}-
$$
$$
-\frac{1}{s^{c_1+\dots+c_{n-1}+p+1+c_{n+1}}}\cdot \frac{1}{s-1}=\frac{s^{c_{n+1}}-c_{n+1}(s-1)-1}{(s-1)s^{c_1+\dots+c_{n-1}+p+1+c_{n+1}}} \ge 0.
$$
Here the last inequality is an equality whenever  $ c_{n+1}=1 $.  

\emph{Property 9} follows from Property 8. 

Prove that \emph{Property 10} is true. 
If $x_0 = \Delta^{'} _{c_1c_2...c_n...} \in \mathbb  S_{(s,0)}$, then 
$$
x_0  = \Delta^{'} _{c_1c_2...c_n...} \in \Delta^{'} _{c_1} \cap  \Delta^{'} _{c_1c_2} \cap \ldots \cap  \Delta^{'} _{c_1c_2...c_n}\cap \dots .
$$
Hence, 
$$
x_0  \in I_{c_1} \cap  I _{c_1c_2} \cap \ldots \cap  I _{c_1c_2...c_n}\cap \dots .
$$
From properties of cylinders it follows that 
$$
 I_{c_1} \supset  I _{c_1c_2} \supset \ldots \supset  I _{c_1c_2...c_n} \supset \dots .
$$
Therefore,
$$
  x_0=\bigcap_ {n=1}^{\infty} {\Delta' _{c_1c_2...c_n}}. 
$$
\end{proof}

\begin{corollary} The condition  
$$
\mathbb  S_{(s,0)} \subset I_0= \left[\frac{s-1}{s^{s-1}-1};\frac{1}{s-1}\right].
$$
is satisfied for any positive integer $s>2$.
\end{corollary}
\begin{theorem}
The set $\mathbb  S_{(s,0)}$ is a perfect and nowhere dense set of zero Lebesgue measure.
\end{theorem}
\begin{proof}
  Let us prove that \emph{the set $\mathbb  S_{(s,0)}$ is a  nowhere dense set}.  From the definition it follows that there exist cylinders $ \Delta^{'} _{c_1...c_n}$ of rank $n$ in an arbitrary subinterval of the segment    $ I_0 $. Since Property 7 and Property 8 hold for these cylinders, we have that for any subinterval of  $ I_0 $ there exists a subinterval such that does not contain points from $\mathbb  S_{(s,0)}$. So $\mathbb  S_{(s,0)}$ is a  nowhere dense set.
 
Show that \emph{$\mathbb  S_{(s,0)}$ is a set of zero Lebesgue measure}. Suppose that 
$$
 \mathbb  S_{(s,0)}= \bigcap^{\infty} _{k=1} E_k,
$$
where
$$
E_1= I_1 \cup I_2 \cup\ldots \cup I_{s-1},
$$
$$
E_2= I_{11} \cup \ldots\cup I_{[s-1][s-1]},
$$
$$
E_3= I_{111} \cup \ldots \cup I_{[s-1][s-1][s-1]},
$$
$$
......................................................
$$
$$
E_k= I _{\underbrace{1...1}_{k}} \cup \ldots\cup I _{\underbrace{[s-1]...[s-1]}_{k}},
$$
$$
......................................................
$$
and $ I_{c_1c_2...c_n} $ is a closed interval whose endpoints coincide with endpoits of the cylinder $ \Delta^{'} _{c_1c_2...c_n} $.
In addition, since $ E_{k+1} \subset E_k $, we have 
$$
E_k= E_{k+1} \cup \bar E_{k+1}.
$$

Let $  I_0 $ be an initial closed interval such that $ \lambda(I_0)=d_0 $ and $\mathbb  S_{(s,0)}\subset I_0$. Then
$$
\lambda(E_1)=\left(\frac{1}{s}+\frac{1}{s^2}+\frac{1}{s^3}+\dots+\frac{1}{s^{s-1}}\right)d_0.
$$
If
$$
\frac{1}{s}+\frac{1}{s^2}+\frac{1}{s^3}+\dots+\frac{1}{s^{s-1}}=\sigma, 
$$
then
$$
\lambda(\bar E_1)=d_0 - \sigma d_0= d_0(1 - \sigma).
$$

Similarly,
$$
\lambda(\bar E_2)=d^{'(1)} _{i} - \sigma d^{'(1)} _{i}=d^{'(1)} _{i}(1 - \sigma)=\left(\frac{1}{s}+\frac{1}{s^2}+\dots+\frac{1}{^{s-1}}\right)d_0(1 - \sigma)=
$$
$$
= \sigma d_0 (1 - \sigma), ~\mbox{where}~ i=\overline{1,s-1}, 
$$
and 
$$
\lambda(\bar E_3)=d^{'(2)} _{i} - \sigma d^{'(2)} _{i}=d^{'(2)} _{i}(1 - \sigma)=\sigma d^{'(1)} _{i}(1 - \sigma)=\sigma^2 d_0(1 - \sigma).
$$
So we obtain that the sequence
$$
d_0(1 - \sigma),\sigma d_0 (1 - \sigma),\sigma^2 d_0(1 - \sigma),\dots ,\sigma^{n-1} d_0(1 - \sigma), \dots ,
$$
is an infinitely decreasing geometric progression. Hence, 
$$
\lambda(\mathbb  S_{(s,0)})= d_0 - \sum^{\infty} _{k=1} \lambda(\bar E_k)=d_0 - \sum^{\infty} _{j=1} {\sigma^{j-1}d_0(1 - \sigma)}=d_0 - \frac{d_0(1 - \sigma)}{1-\sigma}=0.
$$
So $\mathbb  S_{(s,0)}$  is a set of zero Lebesgue measure. 

Prove that \emph{$\mathbb  S_{(s,0)}$  is a perfect set}. Since $ E_k= I _{\underbrace{1...1}_{k}} \cup \ldots\cup I _{\underbrace{[s-1]...[s-1]}_{k}} $ is a closed sets ($E_k$ is a union of segments), we see that 
$$
 \mathbb  S_{(s,0)}= \bigcap^{\infty} _{k=1} E_k
$$
is a closed set. 

Let $ x \in \mathbb  S_{(s,0)} $,    $ P $  be any interval that contains $ x $, and $ J_n $ be a segment of  $ E_n $ that contains $ x $. Choose a number $ n $ such that $  J_n \subset P $. Suppose that $ x_n $ is the endpoint of $ J_n $ such that the condition 
$ x_n \ne x $ holds. Hence $ x_n \in \mathbb  S_{(s,0)} $ and $  x $ is a limit point of the set. 

Since $\mathbb  S_{(s,0)}$ is a closed set and does not contain isolated points, we obtain that $\mathbb  S_{(s,0)}$ is a perfect set. 
\end{proof}

\section{Topological and metric properties of $\mathbb  S_{(s,u)}$, where $u\in A_0$}

By $x_0=\Delta^{(s,u)} _{c_1...c_n...}$ denote the  equality 
$$
x_0=\frac{u}{s-1} +\sum^{\infty} _{k=1}{\frac{c_k - u}{s^{c_1+\dots+c_k}}}.
$$
That is
$$
 x_0 =\Delta^{(s,u)} _{c_1...c_n...} = \Delta^ {s}_{\underbrace{u...u}_{c_1 - 1}c_1 \underbrace{u...u}_{c_2 - 1}c_2...\underbrace{u...u}_{c_n - 1}c_n...}.
$$

\begin{lemma}
The set $\mathbb S_{(s,u)} $ is an uncountable set for an arbitrary $u \in A_0$.
\end{lemma}
\begin{proof}  Define the map $f: \mathbb S_{(s,u)} \to C[s,A_0 \setminus \{u\}] $ by the rule 
$$
\forall (\alpha_n)\in L:x=\left(\frac{u}{s-1}+\sum^{\infty} _{n=1} {\frac{\alpha_n - u}{s^{\alpha_1+ \alpha_2+\dots+\alpha_n}}} \right) \stackrel{f}{\rightarrow} \sum^{\infty} _{n=1} {\frac{\alpha_n}{s^n}}=y=f(x),
$$
i.e., 
$$
x= \Delta^ {s}_{\underbrace{u...u}_{\alpha_1 - 1}\alpha_1 \underbrace{u...u}_{\alpha_2 - 1}\alpha_2...\underbrace{u...u}_{\alpha_n - 1}\alpha_n...}\stackrel{f}{\rightarrow} \Delta^{s} _{\alpha_1\alpha_2...\alpha_n...}=y=f(x).
$$

Assume that there exist $x_1 \ne x_2$ from the set $\mathbb S_{(s,u)} $ such that $f(x_1)=f(x_2)$.  Suppose that 
$x_1=\Delta^ {s}_{\underbrace{u...u}_{\alpha_1 - 1}\alpha_1 \underbrace{u...u}_{\alpha_2 - 1}\alpha_2...\underbrace{u...u}_{\alpha_n - 1}\alpha_n...}$, $x_2=\Delta^ {s}_{\underbrace{u...u}_{\beta_1 - 1}\beta_1 \underbrace{u...u}_{\beta_2 - 1}\beta_2...\underbrace{u...u}_{\beta_n - 1}\beta_n...}$.

If the number $y_{1,2}=f(x_1)=f(x_2)$ has the unique s-adic representation, then $\alpha_n=\beta_n$ for all $n \in \mathbb N$. Hence $x_1=x_2$. The last-mentioned equality contradicts the assumption. 

If $f(x_1)=f(x_2)$  has two different s-adic representations, then there exists $n_0$ such that 
$$
\alpha_{n_0+1}=\alpha_{n_0+2}=\dots=0 \mbox{ or}~ \beta_{n_0+1}=\beta_{n_0+2}=\dots=0.
$$

Since $\alpha_n \ne 0$ and  $\beta_n \ne0$  for any positive integer $n$, we obtain that $f$ is a bijection. Hence  $\mathbb S_{(s,u)} $ is
uncountable, since $ C[s,A_0 \setminus \{u\}] $ is uncountable.
 \end{proof}

\begin{definition} \emph{A cylinder   $ \Delta^{(s,u)} _{c_1...c_n}$ of rank $n$ with base $c_1c_2\ldots c_n$} is a set of the following
form
$$
\Delta^{(s,u)} _{c_1\ldots c_n}=\left\{x: x=\left(\sum^{n} _{k=1} {\frac{c_k -u}{s^{c_1+\dots+c_k}}}\right)+\frac{1}{s^{c_1+\dots+c_n}}{\left( \sum^{\infty} _{i=n+1} {\frac{\alpha_i - u}{s^{\alpha_{n+1}+\dots+\alpha_i}}}\right)}+\frac{u}{s-1}  \right\},
$$
 where $ c_1, c_2,\dots ,c_n $ are fixed s-adic digits,  $\alpha_n \ne u,\alpha_n \ne 0$, and $ 2<s \in\mathbb N, n \in\mathbb N $.
\end{definition}

\begin{lemma} Cylinders $ \Delta^{(s,u)} _{c_1...c_n...} $ have the following properties:
\begin{enumerate}
\item
$$
\inf  \Delta^{(s,u)} _{c_1...c_n...}=\begin{cases}
\tau +\frac{1}{s^{c_1+...+c_n}}\left(\frac{s-1-u}{s^{s-1}-1}+\frac{u}{s-1}\right),&\text{if $u \in \{0,1\}$}\\
$$\\
\tau +\frac{1}{s^{c_1+\dots+c_n}}\frac{1}{s-1},&\text{if $ u \in \{2,3,\dots ,s-1\}$,}
\end{cases}
$$
$$
\sup  \Delta^{(s,u)} _{c_1...c_n...}=\begin{cases}
\tau + \frac{1}{s^{c_1+\dots+c_n}}\frac{1}{s-1},&\text{if $ u=0 $}
$$\\
$$\\
\tau +\frac{1}{s^{c_1+\dots+c_n}}\left(\frac{1}{s^{u+1}-1}+\frac{u}{s-1}\right),&\text{if $ u \in \{1,2,\dots ,s-2\} $}\\
$$\\
\tau +\frac{1}{s^{c_1+\dots+c_n}}\left(1-\frac{1}{s^{s-2}-1}\right),&\text{if $u=s-1$,}\\
\end{cases}
$$
where
$$
\tau= \sum^{n} _{k=1} {\frac{c_k-u}{s^{c_1+\dots+c_k}}}+\sum^{n} _{k=1} {\frac{u}{s^k}}.
$$
\item If $d(\cdot) $ is the diameter of a set, then
$$
d(\Delta^{(s,u)} _{c_1...c_n})=\frac{1}{s^{c_1+\dots+c_n}}d(\mathbb S_{(s,u)});
$$
\item
$$
\frac{d(\Delta^{(s,u)} _{c_1...c_nc_{n+1}})}{d(\Delta^{(s,u)} _{c_1...c_n})}=\frac{1}{s^{c_{n+1}}};
$$
\item 
$$
  \Delta^{(s,u)} _{c_1c_2...c_n} =\bigcup^{s-1} _{i=1} { \Delta^{(s,u)} _{c_1c_2...c_ni}}~~~\forall c_n \in A_0,~~~n \in \mathbb N,~ i \ne u.
$$
\item The following relationships are satisfied: 
\begin{enumerate}
\item if $ u\in \{0,1\}$, then 
$$
\inf \Delta^{(s,u)} _{c_1...c_np}> \sup \Delta^{(s,u)} _{c_1...c_n[p+1]};
$$
\item if  $ u \in \{2,3,\dots ,s-3\}$, then 
$$
\begin{cases}
\sup \Delta^{(s,u)} _{c_1...c_np}< \inf \Delta^{(s,u)} _{c_1...c_n[p+1]}&\text{for all $p+1\le u$}\\
$$\\
\inf \Delta^{(s,u)} _{c_1...c_np}> \sup \Delta^{(s,u)} _{c_1...c_n[p+1]},&\text{for all $u<p$;}
\end{cases}
$$
\item if $ u  \in \{s-2,s-1\}$, then
$$
\sup \Delta^{(s,u)} _{c_1...c_np}< \inf \Delta^{(s,u)} _{c_1...c_n[p+1]}.
$$
\end{enumerate}
\end{enumerate}
\end{lemma}
\begin{proof} \emph{The first property} follows from the equalities: 
$$
\inf \mathbb S_{s,u}=\begin{cases}
\sum^{\infty} _{k=1} {\frac{\max \{\alpha_{k}\}-u}{s^{\max\{\alpha_1\}+\dots+\max \{\alpha_k\}}}+\frac{u}{s-1}},&\text{if $u \in \{0,1\}$}\\
$$\\
\sum^{\infty} _{k=1} {\frac{\min \{\alpha_{k}\}-u}{s^{\min\{\alpha_1\}+\dots+\min \{\alpha_k\}}}+\frac{u}{s-1}},&\text{if $ u \in \{2,3,\dots,s-1\}$,}
\end{cases}
$$
$$
\sup \mathbb S_{(s,u)}=\begin{cases}
\sum^{\infty} _{k=1} {\frac{\min \{\alpha_{k}\}}{s^{\min\{\alpha_1\}+\dots+\min \{\alpha_k\}}}},&\text{if $ u=0 $}
$$\\
$$\\
\sum^{\infty} _{k=1}{\frac{1}{s^{(1+u)k}}}+\frac{u}{s-1},&\text{if $ u \in \{1,2,\dots,s-2\} $}\\
$$\\
\sum^{\infty} _{k=1} {\frac{\max \{\alpha_{k}\}-s+1}{s^{\max\{\alpha_1\}+\dots+\max \{\alpha_k\}}}+1},&\text{if $u=s-1$.}\\
\end{cases}
$$
These equalities follow from the definition of  $ \mathbb S_{s,u}$. 

It is easy to see that \emph{the second property} follows from the first property, \emph{the third property} is a corollary of the second property, and  \emph{Property 4} follows from the definition. 

Let us show that \emph{Property 5} is true. We now prove that the first inequality holds for  $ u=1$. In fact,
$$
\inf \Delta^{(s,1)} _{c_1...c_np}- \sup \Delta^{(s,1)} _{c_1...c_n[p+1]}=\left(\sum^{n} _{k=1} {\frac{c_k-1}{s^{c_1+\dots+c_k}}}\right)+\left(\sum^{c_1+\dots+c_n} _{k=1} {\frac{1}{s^k}}\right)+
$$
$$
+\frac{1}{s^{c_1+\dots+c_n+1}}+\frac{1}{s^{c_1+\dots+c_n+2}}+\dots+\frac{1}{s^{c_1+\dots+c_n+p-1}}+\frac{p}{s^{c_1+\dots+c_n+p}}+\frac{\inf \mathbb S_{(s,1)}}{s^{c_1+\dots+c_n+p}}-
$$
$$
-\left(\sum^{n} _{k=1} {\frac{c_k-1}{s^{c_1+\dots+c_k}}}\right)-\left(\sum^{c_1+\dots+c_n} _{k=1} {\frac{1}{s^k}}\right)-\frac{1}{s^{c_1+\dots+c_n+1}}-\frac{1}{s^{c_1+\dots+c_n+2}}-\dots-
$$
$$
-\frac{1}{s^{c_1+\dots+c_n+p-1}}-\frac{1}{s^{c_1+\dots+c_n+p}}-\frac{p+1}{s^{c_1+\dots+c_n+p+1}}-\frac{\sup \mathbb S_{(s,1)}}{s^{c_1+\dots+c_n+p+1}}=
$$
$$
=\frac{p}{s^{c_1+\dots+c_n+p}}+\frac{1}{s^{c_1+\dots+c_n+p}}\left(\frac{s-2}{s^{s-1}-1}+\frac{1}{s-1}\right)-\frac{1}{s^{c_1+\dots+c_n+p}}-
$$
$$
-\frac{p+1}{s^{c_1+\dots+c_n+p+1}}-\frac{1}{s^{c_1+\dots+c_n+p+1}}\left(\frac{1}{s^2-1}+\frac{1}{s-1}\right)=
$$
$$
=\frac{1}{s^{c_1+\dots+c_n+p}}\left(\left(p-1-\frac{p+1}{s}\right)+\left(\frac{(s^{s-1}-1)((s+1)s-s-2)+(s-2)(s^2-1)s}{(s-1)s(s+1)(s^{s-1}-1)}\right) \right)>0,
$$
where $1< p<s-1$ and $ s>2 $.

Consider the system of inequalities. Let us prove that the first inequality of the system is true. Here  $ p+1\le u $. In fact,
$$
\sup \Delta^{(s,u)} _{c_1...c_np}- \inf \Delta^{(s,u)} _{c_1...c_n[p+1]}=\frac{p}{s^{c_1+\dots+c_n+p}}+\frac{1}{s^{c_1+\dots+c_n+p}}\left(\frac{1}{s^{u+1}-1}+\frac{u}{s-1}\right)-
$$
$$
-\frac{u}{s^{c_1+\dots+c_n+p}}-\frac{p+1}{s^{c_1+\dots+c_n+p+1}}-\frac{1}{(s-1)s^{c_1+\dots+c_n+p+1}}=
$$
$$
=\frac{1}{s^{c_1+\dots+c_n+p}}\left(\frac{1}{s^{u+1}-1}+\frac{u}{s-1}-u-\frac{p+1}{s}-\frac{1}{s(s-1)}+p\right)<
$$
$$
<\frac{1}{s^{c_1+\dots+c_n+p}}\left(\frac{1}{s^{u+1}-1}-\frac{1}{s(s-1)}+\frac{u}{s-1}-\frac{p+1}{s}-1\right)<0.
$$

Prove that the second inequality is true. Here $ p>u $, i.e.,  $p-u \ge 1$. Similarly,
$$
\sup \Delta^{(s,u)} _{c_1...c_n[p+1]}-\inf \Delta^{(s,u)} _{c_1...c_np}=\frac{p}{s^{c_1+\dots+c_n+p}}+\frac{1}{s^{c_1+\dots+c_n+p}}\frac{1}{s-1} -\frac{u}{s^{c_1+\dots+c_n+p}}-
$$
$$
-\frac{p+1}{s^{c_1+\dots+c_n+p+1}}-\frac{1}{s^{c_1+\dots+c_n+p+1}}\left(\frac{1}{s^{u+1}-1}+\frac{u}{s-1}\right)=
$$
$$
=\frac{1}{s^{c_1+\dots+c_n+p}}\left(p+\frac{1}{s-1}-u-\frac{p+1}{s}-\frac{1}{s(s^{u+1}-1)}-\frac{u}{s(s-1)}\right)\ge
$$
$$
\ge \frac{1}{s^{c_1+\dots+c_n+p}}\left(1+\frac{1}{s-1}-\frac{p+1}{s}-\frac{1}{s(s^{u+1}-1)}-\frac{u}{s(s-1)}\right)>0.
$$

The last inequality of Property 5  is a corollary of the system of inequalities and is proved by analogy. 
\end{proof}
\begin{corollary} The following equalities are hold: 
$$
\inf \mathbb S_{s,u}=\begin{cases}
\sum^{\infty} _{k=1} {\frac{s-1-u}{{s^{(s-1)k}}}+\frac{u}{s-1}}=\frac{s-u-1}{s^{s-1}-1}+\frac{u}{s-1},&\text{if $u \in \{0,1\}$}\\
$$\\
\sum^{\infty} _{k=1} {\frac{1-u}{{s^k}}+\frac{u}{s-1}}=\frac{1}{s-1},&\text{if $ u \in \{2,3,\dots,s-1\}$,}
\end{cases}
$$
$$
\sup \mathbb S_{(s,u)}=\begin{cases}
\frac{1}{s-1},&\text{if $ u=0 $}
$$\\
$$\\
\sum^{\infty} _{k=1}{\frac{1} {s^{(u+1)k}}}+\frac{u}{s-1}=\frac{1}{s^{u+1}-1}+\frac{u}{s-1},&\text{if $ u \in \{1,2,\dots,s-2\} $}\\
$$\\
1-\frac{1}{s^{s-2}-1},&\text{if $u=s-1$.}\\
\end{cases}
$$
\end{corollary}
\begin{theorem}
The set $\mathbb S_{(s,u)}$ is a perfect and nowhere dense set of zero Lebesgue measure.
\end{theorem}
\begin{proof} Since $\mathbb S_{(s,u)} \subset C[s,A_0]$ for all $u=\overline{1,s-1}$, where $ A_0=\{1,2,...,s-1\} $, we see that $\mathbb S_{(s,u)}$ is a nowhere dense set of zero Lebesgue measure. The fact that $\mathbb S_{(s,u)}$ is a perfect set can be proved by analogy with the proof of this fact for $\mathbb S_{(s,0)}$. 
\end{proof}

\section{Fractal properties of  $\mathbb  S_{(s,u)}$}

\begin{theorem}
The set  $\mathbb S_{(s,u)} $ is a self-similar fractal and the Hausdorff-Besicovitch dimension $\alpha_0 (\mathbb S_{(s,u)})$ of this set satisfies the following equation 
$$
\sum _{p_i \ne u, p_i \in A_0} {\left(\frac{1}{s}\right)^{p_i \alpha_0}}=1.
$$
\end{theorem}
\begin{proof}
Consider the set $\mathbb S_{(s,0)}$. Since 
$ \mathbb S_{(s,0)} \subset I_0 $ and $ \mathbb S_{(s,0)}$ is a perfect set, we obtain that  $\mathbb S_{(s,0)}$ is a compact set. In addition, 
$$
\mathbb S_{(s,0)}=[I_1 \cap \mathbb S_{(s,0)}]\cup [I_2 \cap \mathbb S_{(s,0)}]\cup\ldots\cup[I_{s-1}\cap \mathbb S_{(s,0)}]
$$
and
$$
 {[I_1 \cap \mathbb S_{(s,0)}]} \stackrel{s^{-1}}{\sim}\mathbb S_{(s,0)}, {[I_2 \cap \mathbb S_{(s,0)}]} \stackrel{s^{-2}}{\sim}\mathbb S_{(s,0)}, \dots, {[I_{s-1} \cap \mathbb S_{(s,0)}]} \stackrel{s^{-(s-1)}}{\sim}\mathbb S_{(s,0)}.
$$
Since the set $\mathbb S_{(s,0)}$ is a compact self-similar set of space  $ \mathbb R^1 $, we have that the self-similar dimension of this set is equal to the Hausdorff-Besicovitch dimension of $\mathbb S_{(s,0)}$. So the set  $\mathbb S_{(s,0)} $ is a self-similar fractal, and its Hausdorff-Besicovitch dimension $\alpha_0(\mathbb S_{(s,0)})$  satisfies the equation
$$
\left (\frac{1}{s}\right)^{\alpha_0}+\left (\frac{1}{s}\right)^{2\alpha_0}+\left (\frac{1}{s}\right)^{3\alpha_0}+\dots+\left (\frac{1}{s}\right)^{(s-1)\alpha_0}=1.
$$

Fractal properties  of the set $\mathbb S_{(s,u)}$ ($u\in A_0$) can be formulated and  proved by analogy with the case of $\mathbb S_{(s,0)} $. However in the case of $\mathbb S_{(s,u)}$, the condition $\alpha_n \ne u$ for all $n \in \mathbb N$
should be taken into account.
\end{proof}

Consider the set of all numbers whose s-adic representations contain only combinations of s-adic digits that is using  in the s-adic representations of elements of $\mathbb S_{(s,u)}$.

\emph{ A set $  \tilde S $ } is the set of all numbers whose s-adic representations  contain only combinations of s-adic digits from the set
$$
\{1, 02, 003,\dots, {\underbrace{u...u}_{c-1}}c, \dots, {\underbrace{(s-1)...(s-1)}_{s-3}}(s-2)\},
$$
where $c \in A_0, u \in A$, $c \ne u$.

\begin{theorem}
The set $ \tilde S $ is an uncountable, perfect, and nowhere dense set of zero Lebesgue measure.
\end{theorem}

Properties of the set  $ \tilde S $ follow from the following fractal properties of this set. 

\begin{theorem}
 The set $ \tilde S $ is a self-similar fractal, and its Hausdorff-Besicovitch dimension $\alpha_0$ satisfies the following equation 
$$
\left(\frac{1}{s}\right)^{\alpha_0}+(s-1)\left(\frac{1}{s}\right)^{2\alpha_0}+(s-1)\left(\frac{1}{s}\right)^{3\alpha_0}+\dots+(s-1)\left(\frac{1}{s}\right)^{(s-1)\alpha_0}=1.
$$
\end{theorem}
\begin{proof} 
The s-adic representation of an arbitrary element from $ \tilde S $ contains combinations of digits from the following  tuple:
$$
02, 003,\dots,\underbrace{0\ldots 00}_{s-2}(s-1);
$$
$$
1, 12, 113,\dots, \underbrace{1\ldots 11}_{s-2}(s-1);
$$
$$
223, 2224,\dots, \underbrace{2\ldots 22}_{s-2}(s-1);
$$
$$
\dots \dots \dots \dots \dots \dots \dots
$$
$$
u2, uu3,\dots,\underbrace{u\ldots uu}_{u-2}(u-1), \underbrace{u\ldots uu}_{u}(u+1), \dots, \underbrace{u\ldots uu}_{s-2}(s-1);
$$
$$
\dots \dots \dots \dots \dots \dots \dots
$$
$$
(s-1)2, (s-1)(s-1)3,\dots, \underbrace{(s-1)\ldots (s-1)(s-1)}_{s-3}(s-2).
$$
Here $s^2 -3s+3$ combinations of s-adic digits, i.e., the unique 1-digit combination  and $s-1$ k-digit combinations for all $k=\overline{2,s-1}$. Our statement follows from  Theorem \ref{th: theorem6}.
\end{proof}

\begin{theorem}
\label{th: theorem6}
Let  $E$ be a set whose elements represented by finite number of fixed combinations $\sigma_1, \sigma_2,\dots,\sigma_m$ of s-adic digits in the s-adic number system. Then the Hausdorff-Besicovitch dimension $\alpha_0$ of $E$ satisfies the following equation: 
$$
N(\sigma^1 _m)\left(\frac{1}{s}\right)^{\alpha_0}+N(\sigma^2 _m)\left(\frac{1}{s}\right)^{2\alpha_0}+\dots+N(\sigma^{k} _m)\left(\frac{1}{s}\right)^{k\alpha_0}=1,
$$
where $N(\sigma^k_m)$ is a number of k-digit combinations $\sigma^k_m$ from the set $\{\sigma_1, \sigma_2,\dots,\sigma_m\}$,
$k \in \mathbb N$, and $N(\sigma^1 _m)+N(\sigma^2 _m)+\dots+ N(\sigma^{k} _m)=m$.
\end{theorem}
\begin{proof} Let  $\{\sigma_1, \sigma_2,\dots,\sigma_m\}$ be a set of fixed combinations of s-adic digits, and the s-adic representation of any number from $E$ contains only such combinations of digits. 

It is easy to see that there exist combinations $\sigma', \sigma''$ from the set $\{\sigma_1, \sigma_2,\dots,\sigma_m\}$ such that 
$\Delta^s _{\sigma^{'}\sigma^{'}...}=\inf E$, $\Delta^s _{\sigma^{''}\sigma^{''}...}=\sup E$,
and
$$
d(E)=\sup E - \inf E=\Delta^s _{\sigma^{''}\sigma^{''}...}-\Delta^s _{\sigma^{'}\sigma^{'}...}.
$$

\emph{A cylinder $ \Delta^{(s,E)} _{\sigma_1\sigma_2\ldots\sigma_n}$ of rank $n$ with base $\sigma_1\sigma_2\ldots\sigma_n$} is a set  formed by all numbers of $E$ with s-adic representations in which the first $n$ combinations of digits are fixed and  coincide with $\{\sigma_1,\sigma_2,\dots,\sigma_n\}$ respectively ($\sigma_j\in \{\sigma_1, \sigma_2,\dots,\sigma_m\}$ for all $j=\overline{1,n}$).

It is easy to see that
$$
d( \Delta^{(s,E)} _{\sigma_1\sigma_2...\sigma_n})=\frac{1}{s^{N(\sigma_1+\sigma_2+\dots+\sigma_n)}} d(E),~\mbox{де}
$$
where $N(\sigma_1+\sigma_2+\dots+\sigma_n)$ is a number of s-adic digits in the combination  $\sigma_1\sigma_2\ldots\sigma_n$. 

Since  $E$ is a Cantor type set, $ E \subset [\inf E; \sup E] $, and
$$
\frac{d( \Delta^{(s,E)} _{\sigma_1\sigma_2...\sigma_n\sigma_{n+1}})}{d( \Delta^{(s,E)} _{\sigma_1\sigma_2...\sigma_n})}=\frac{s^{N(\sigma_1+\sigma_2+\dots+\sigma_n)}}{s^{N(\sigma_1+\sigma_2+\dots+\sigma_n+\sigma_{n+1})}}=\frac{1}{s^{N(\sigma_{n+1})}},
$$
$$
E=[I_{\sigma_1} \cap E]\cup [I_{\sigma_2} \cap E]\cup\ldots\cup[I_{\sigma_m}\cap E],
$$
where
$I_{\sigma_i}=[\inf \Delta^{(s,E)} _{\sigma_i};\sup \Delta^{(s,E)} _{\sigma_i}]$ and $i=1,2,\dots,m,$
we have
$$
 {[I_{\sigma^1 _m} \cap E]} \stackrel{s^{-1}}{\sim}E, {[I_{\sigma^2 _m} \cap S]} \stackrel{s^{-2}}{\sim}E, {[I_{\sigma^3 _m} \cap E]} \stackrel{s^{-3}}{\sim}E,\dots, {[I_{\sigma^k _m} \cap E]} \stackrel{s^{-k}}{\sim}E.
$$
This completes the proof.
\end{proof}

\begin{remark} The following statements are true:
\begin{itemize}
\item  if there exists $k \in \mathbb N$ such that $N(\sigma^k _m)=s^k$, then $\alpha_0 (E)=1$;
\item if $m=1$, then $\alpha_0 (E)=0$.
\end{itemize}
\end{remark}

\section{Normal numbers on $\bigcup^{s-1} _{u=0} {\mathbb S_{(s,u)}}$}

Suppose that
$$
\mathbb  S_0=\bigcup^{s-1} _{u=0} {\mathbb S_{(s,u)}}.
$$

 \emph{The frequence of a digit  $i$ in the s-adic expansion of a number $x$} is a limit 
$$
\upsilon_{i}=\lim_{k \to \infty} {\frac{N_i (x,k)}{k}}~\mbox{(if there exists a limit from the right)},
$$
where $N_i (x,k)$ is a number of the digit $i$ $(i=\overline{0,s-1})$ up to the k-th place inclusively in the s-adic expansion, and 
$k^{-1}N_i (x,k)$ is \emph{the relative frequence of using of the digit $i$} in the s-adic expansion of $x$. 
 
A number $x \in [0;1]$  is called \emph{a normal number on the base $s$} whenever for each  $i \in \{0, 1,\dots, s-1\}$ there exists
the frequence and $\upsilon_{i}(x)=s^{-1}$.

\begin{theorem} Let $s>2$ be a fixed positive integer; then the following statements are true:
\begin{itemize}
\item if $s=3$, then the set of normal numbers from $ \mathbb  S_0 $  is a continuum set (a subset of $\mathbb S_{(3,0)}$), and its Hausdorff-Besicovitch dimension $\alpha_0$ satisfies the following inequalities
$$
\frac{1}{3}\log_3 2\le  \alpha_0 \le~\log_3 {\frac{2}{\sqrt{5}-1}}; 
$$
\item if $s\ne 3$, then the set of normal numbers from $ \mathbb  S_0 $  is an empty set.
\end{itemize}
\end{theorem}

\begin{proof} If a number $x$ does not belong to the set of normal numbers from $ \mathbb  S_0$, then there exists sequence (at least one) of digits of $x$ such that the frequence of one of the digits does not exist or is not equal to $ s^{-1}$. Find such sequence $(k_n)$. 
 
Suppose that $k_n=ns$, i.e.,  choose $ns$ digits. Here $n \in \mathbb N$, $s \ge 3$, and 
$$
\bigtriangleup k=k_{n+1}-k_n=ts,
$$
where $t$ is a fixed positive integer.

Let $x$ be a normal number. Then 
$$
\lim_{k_{n+1 }\to \infty} {\frac{N_{1}(x,k_{n+1})}{k_{n+1}}}=\lim_{k_{n+1 }\to \infty} {\frac{N_{2}(x,k_{n+1})}{k_{n+1}}}=\dots=\lim_{k_{n+1 }\to \infty} {\frac{N_{s-1}(x,k_{n+1})}{k_{n+1}}}=\frac{1}{s},
$$
where $k_{n+1}=(n+t)s$. Hence, 
$$
N_{1}(x,k_{n+1}) \to (n+t),
$$
$$
N_{2}(x,k_{n+1}) \to (n+t),
$$
$$
\dots\dots\dots
$$
$$
N_{s-1}(x,k_{n+1}) \to (n+t).
$$

Find $N_{0}(x,k_{n+1})$ by the definition of $\mathbb S_{(s,u)}$. Since $N_{0}(x,k_{n+1}) \ne 0$ holds  only for $\mathbb S_{(s,0)}$  (i.e., the set of normal numbers of $ \mathbb  S_0$ is a subset of $\mathbb S_{(s,0)}$), we obtain 
$$
N_{0}(x,k_{n+1}) \to A=N_{2}(x,k_{n+1})+2N_{3}(x,k_{n+1})+\dots+(s-2)N_{s-1}(x,k_{n+1}),
$$
$$
A= \sum^{s-2} _{j=1} {[j(n+t)]}=\frac{(s-2+1)(s-2)(n+t)}{2}=\frac{(s-2)(s-1)(n+t)}{2}.
$$
Therefore,
$$
\lim_{k_{n+1 }\to \infty} {\frac{N_{0}(x,k_{n+1})}{k_{n+1}}}=\lim_{(n+t)\to \infty}{\frac{(s-2)(s-1)(n+t)}{2s(n+t)}}=\frac{(s-2)(s-1)}{2s} \ne \frac{1}{s}
$$
for any $ s>3$.  If $s=3$, then we have the equality. 

By $S_N$ denote the set of normal numbers of $\mathbb  S_0$. It is easy to see that
$$
\alpha_0(S^{'} _3)\le\alpha_0\left(S_N\right)\le\alpha_0\left(\mathbb S_{(3,0)}\right), 
$$
where $S^{'} _3=\{x: x=\Delta^3 _{\sigma_1\sigma_2\ldots\sigma_n\ldots}, \sigma\in\{021, 102\}\}$. 
It follows from Theorem \ref{th: theorem6} that $\alpha_0\left(\mathbb S_{(3,0)}\right)=\log_3 {\frac{2}{\sqrt{5}-1}}$ and the Hausdorff-Besicovitch  dimension of $S^{'} _3$ satisfies the equation 
$$
2\left(\frac{1}{3}\right)^{3\alpha_0}=1.
$$
So $\alpha_0(S^{'} _3)=\frac{1}{3}\log_3 2.$
\end{proof}

\end{document}